\documentclass[11pt,a4paper,reqno]{amsart}
  \usepackage{amsfonts,amsmath,amssymb,amsopn,amsthm}
  \usepackage[usenames,dvipsnames]{xcolor}
  \usepackage[square,sort&compress,comma,numbers]{natbib}
  \usepackage{framed}
  \usepackage{soul}
  \usepackage{tikz}
  \usepackage{wasysym}
  \usetikzlibrary{patterns}
  \usepackage{a4wide}
 
\newcommand{\thistime}{\expandafter\calctimeA\pdfcreationdate\@nil}
\def\calctimeA#1:#2#3#4#5#6#7#8#9{\calctimeB}
\def\calctimeB#1#2#3#4#5\@nil{#1#2:#3#4}

   \def\R{\mathbb{R}}
   \def\N{\mathbb{N}}
   
   \def\1{{\rm I\mskip -10.5mu 1}}
   
   \def\e{{\varepsilon}}
   \def\D{{\nabla}}

   \def\cA{{\mathcal A}}
   \def\cB{{\mathcal B}}
   \def\cC{{\mathcal C}}
   \def\cD{{\mathcal D}}

   \def\cG{{\mathcal G}}

   \def\cL{{\mathcal L}}

   \def\cP{{\mathcal P}}

   \def\supp{\mathop{\rm supp}\nolimits}

   \def\loc{\mathop{\rm loc}\nolimits}

\theoremstyle{definition}
\newtheorem{df}{Definition}[section]
\theoremstyle{remark}
\newtheorem{rem}[df]{Remark}
\theoremstyle{plain}
\newtheorem{prop}[df]{Proposition}
\newtheorem{lemma}[df]{Lemma}
\newtheorem{teo}[df]{Theorem}

\newtheorem{cor}[df]{Corollary}

\newcommand{\beq}{\begin{equation}}
\newcommand{\eeq}{\end{equation}}

 \newcommand{\sezione}[1]{\section{#1}\setcounter{equation}{0}}

\begin{document}


\title[]{Positive solutions to nonlinear elliptic problems involving Sobolev exponent}

\author{Carlo Mercuri}
\address{Dipartimento di Fisica, Matematica e Informatica, Universit\`a degli Studi di Modena e Reggio Emilia, Via Campi, 213 A
41125, Modena Italy }
\email{carlo.mercuri@unimore.it}
\author{Riccardo Molle}
\address{Dipartimento di Matematica, Universit\`a di Roma ``Tor Vergata'',
Via della Ricerca Scientifica n. 1, 00133 Roma, Italy}
\email{molle@mat.uniroma2.it}
\date{\today}


\begin{abstract}
In this paper we consider nonlinear elliptic  PDEs of the type $$-\Delta_p u+a(x)|u|^{p-2}u=|u|^{p^*-2}u \qquad \mbox{ in }\Omega,$$  where $1<p<N$ and $p^*=Np/(N-p)$ is the critical Sobolev exponent, and allowing the asymptotic behavior of the weight function $a$ to be sensitive to the direction. We provide a unified variational approach to obtain existence of distinct solutions in either the unbounded case $\Omega=\R^N$ or when $\Omega$ is a smooth bounded domain.
A key point is a precise description of the compactness properties of certain sequences of approximating solutions (Palais-Smale sequences), for which we use novel observations on nonexistence in certain regimes. Most of our main results are new in the case of the classical Laplace operator, $p=2$.

\end{abstract}


\maketitle


\tableofcontents




\sezione{Introduction}

This paper is devoted to existence, nonexistence and multiplicity of solutions to elliptic PDEs of the form
\beq
\label{P}
\tag{$\cP_{\Omega}$}
-\Delta_p u+a(x)|u|^{p-2}u=|u|^{p^*-2}u, \qquad \mbox{ in }\Omega
\eeq
where either $\Omega=\R^N$ or $\Omega$ is a smooth bounded domain in $\R^N$, $1<p<N$, $p^*=\tfrac{Np}{N-p}$ is the critical Sobolev exponent, $a(x)=a_\infty(x)+\alpha(x)$ where
\beq
\label{b}\tag{$H_a$}
\left\{
\begin{aligned}
&a_\infty\in L^\infty(\R^N),\  \mbox{ with either }\ a_\infty\equiv 0\ \mbox{ or } \\  
&\hspace{2.5cm}  a_\infty(x)=\Theta\left(\tfrac{x}{|x|}\right)\  \forall x\neq 0\mbox{ for some } \Theta\in\cC^1(S^{N-1}),\ \min_{S^{N-1}}\Theta>0\\
&\alpha\in L^{N/p}(\R^N)\setminus\{0\},\ \alpha\ge 0.
\end{aligned}
\right.
\eeq
We seek solutions $u\in X$ within a certain class which is $X=\cD^{1,p}(\R^N)$ if $\Omega=\R^N$ and $a_\infty\equiv 0$, whereas $X=W^{1,p}_0(\Omega)$ otherwise. \\
\noindent The function $\alpha$ will be considered on $\R^N$ even if  $\Omega\neq\R^N$, as our aim is to study the effect of a weight function `$a$' on the topological and variational structure of \eqref{P},  when a possibly anisotropic asymptotic behavior in $\R^N$ is prescribed. In particular we employ certain stable features of the variational framework we use when the PDE is set on $\R^N,$ to deal also with the case of sufficiently large bounded domains $\Omega.$
By scaling, it is easy to see that this analysis could be performed in a slightly different yet equivalent  fashion, setting the PDE on a fixed domain and studying the effect of the weight $\bar{a},$ as it gets more and more `stretched out' (concentrated), see e.g. \cite{BrWi08JFA,Pa96AIHP}. \\
\noindent We tackle problem $(\cP_\Omega)$ by means of variational techniques,
noting that, modulo dilations, solutions are critical points of the action functional
$$
F_a(u)=\int_{\Omega}\Big( {|\nabla u |^p+a(x)|u|^p}  \Big)dx,\qquad u\in X,
$$
constrained on the unit $L^{p^*}$-sphere
$$
S_{p^*}=\left\{u\in X\ :\ \int_{\R^N}|u(x)|^{p^*}\, dx=1\right\}.
$$
However, by the sign condition on $a,$ solutions to \eqref{P} cannot be found as {\it groundstates}, hence a variational scheme more subtle than direct minimisation is needed. In fact, we have the following
\begin{prop}
\label{NEM}
Let the function $a$ be as in \eqref{b}. Then, for every $\Omega\subseteq\R^N$, 
\beq
\label{eqmin}
\inf_{S_{p^*}}F_a=S
\eeq
and this infimum is not attained.
\end{prop}
\noindent Hereby we denote by $S$ the best Sobolev constant, that is
\beq
\label{Sobolev}\tag{$\mathcal{S}$}
S=\inf_{u\in W^{1,p}(\R^N)\setminus\{0\}} \frac{\|u\|_{W^{1,p}}^p}{\|u\|_{L^{p^*}}^p}
 =\inf_{u\in \cD^{1,p}(\R^N)\setminus\{0\}}\frac{\|u\|_{\cD^{1,p}}^p}{\|u\|_{L^{p^*}}^p}.
\eeq
The proposition above, which is well-known in the case $p=2$  and for constant $a_\infty$ from \cite[Proposition 2.8]{CM19ESAIM}, is briefly justified in Section \ref{compactnesssection}. \\
\noindent  Problems of type $(\cP_\Omega)$ have been extensively studied since the work of Br\'ezis and Nirenberg \cite{BrNi83CPAM}, mainly in the semi-linear case $p=2$.
  In a pioneering paper, Benci and Cerami  have shown in \cite{BeCe90JFA} that if $\|a\|_{N/2}$ is sufficiently small then $(\cP_{\R^N})$ possesses at least a positive bound state solution.
 Their proofs are based on a precise characterisation of the possible loss of compactness of the sequences of approximated critical points of an action functional associated with \eqref{P} (see also  \cite{CeSoSt86JFA,St84MZ} for $p=2$ and \cite{MR1926623,FaMeWi19CVPDE,MeWi10DCDS} for $p\neq 2$). Our topological argument for existence is inspired by their work in combining degree theory and tools such as the `barycenter' and the `momentum' of a function, which we review in our new setting.\\
\noindent  When $a(x)\equiv a_\infty\in (0,\infty)$ it is well-known that $(\cP_\Omega)$ has no solution if $\Omega=\R^N$ or if it is a starshaped domain, by the Pohozaev identity.
A nonexistence result when the potential is non-autonomous is shown in \cite{CiPi04ZAMP}.
More precisely, if the potential has the form $\lambda\,  a(x)$, then there are no solutions concentrating at points, as $\lambda\to\infty$. \\
\noindent Still dealing with the case $p=2$ and for critical growth nonlinearities, in \cite{CM19ESAIM} Cerami and the second author studied a more general setting which also includes the so-called Schr\"odinger-Poisson systems. By means of a compactness analysis which covers also the case $a\sim a_\infty\in(0,\infty)$ at infinity,
  it is shown that $(\cP_{\R^N})$ may possess either one or two distinct positive solutions, in the case of potentials that are allowed to converge to a positive constant at infinity.
 These ideas, introduced and developed in \cite{BeCe90JFA,CM19ESAIM} have been used to deal with an interesting variety of semilinear problems, see e.g. \cite{AFM21DCDS,GLLM22arXiv,GaSiYaZh20PRSE}. \\
\noindent It is worth mentioning that critical problems of the form $(\cP_{\R^N})$ are also studied by variants of the Lyapunov–Schmidt reduction method, yielding 
existence of infinitely many solutions. This is the case in \cite{PeWaYa18JFA} (see also references therein), where suitable potentials radially symmetric in $N-1$-dimensions have been considered.\\

\bigskip

When $a_\infty$ is non-constant, much less has been explored, for $p=2$ as well as in the quasilinear setting $p\neq 2.$

\bigskip

The aim of this paper is to tackle $(\cP_\Omega)$ by means of a unified variational approach for $p\in(1,N),$
which is robust enough to deal with either the case $\Omega=\R^N$ or when $\Omega\subset\R^N$ is large and bounded, and where the asymptotic behavior of the weight function $a$ at infinity is allowed to vary with the direction. As far as we know, most of our main results are new also for $p=2$. This is the case of the following observation on nonexistence, which plays a role in understanding the compactness features of the Palais-Smale sequences associated with \eqref{P}.

\begin{teo}
\label{TNE}
Let $a_\infty\not\equiv 0$ satisfy \eqref{b} and $u\in W^{1,p}(\R^N)$ be a weak solution to
\beq
\label{Pinfty}
\tag{$\cP_{a_\infty}$}
-\Delta_p u+a_\infty(x)|u|^{p-2}u=|u|^{p^*-2}u \qquad \mbox{ in }\R^N.
\eeq
Then, $u\equiv 0$. 
\end{teo}
Our proof of this Liouville-type fact is in the spirit of that of classical variational identities suitable for unbounded domains developed in the semilinear case by Esteban and Lions in \cite{MR0688279}, and in a quasilinear setting by Pucci and Serrin \cite{MR0855181}, and takes into account the more delicate regularity theory in the case $p\neq 2.$ \newline  
\noindent We believe that by an approximation argument our regularity assumption on $\Theta$ may be relaxed removing completely any assumptions on its derivatives, as these are not involved in the Pohozaev-type identity we use.  \\

Our main existence and multiplicity results can be stated as follows.
\begin{teo}
\label{R^N}
Assume that the potential $a=a_\infty+\alpha$ satisfies assumptions \eqref{b}.
\begin{itemize}
\item[{\em (i)}]
If $a_\infty\equiv 0$, then there exists $L>0$ such that if 
\beq
\label{1.1}
\|\alpha\|_{N/p}<L
\eeq  
then  $(\cP_{\R^N})$ has a nontrivial nonnegative solution $u_h$.
\item[{\em (ii)}]
There exists $M_1=M_1(\alpha)>0$  such that if
\beq
\label{M1}
 \|a_\infty\|_\infty\in(0,M_1)
\eeq
then $(\cP_{\R^N})$  has a nontrivial nonnegative solution $u_l$.
\item[{\em (iii)}]
If $\alpha$ verifies $\|\alpha\|_{N/p}<L$, then there exists $M_2=M_2(\alpha)>0$ such that if 
\beq\label{1.3}
 \|a_\infty\|_\infty\in(0,M_2)
\eeq 
then  $(\cP_{\R^N})$ has two distinct nontrivial nonnegative solutions $u_l,u_h$.
\end{itemize}
\end{teo}

\medskip

\begin{teo}
\label{bounded}
Let $\Omega$ be a bounded domain in $\R^N$ and  $a$ verifies \eqref{b}.
Then
\begin{itemize}
\item[{\em (i)}]
if $\|a_\infty\|_\infty\in [0,M_1)$  (see \eqref{M1}), there exists $R_1=R_1(\alpha,\|a_\infty\|_\infty)>0$  such that if $B_{R_1}(0)\subseteq \Omega$ 
then  \eqref{P} has a nontrivial nonnegative solution $u_l$;
\item[{\em (ii)}]
if $\|\alpha\|_{N/p}<L$ and $\|a_\infty\|_\infty\in [0,M_2)$ (see \eqref{1.1} and \eqref{1.3}),  there exists   $R_2=R_2(\alpha)>0$  such that if $B_{R_2}(0)\subseteq \Omega$
then  \eqref{P} has two nontrivial nonnegative solutions $u_l$ and $u_h$.
\end{itemize}
\end{teo}
\begin{rem}
It is standard to see that any slightly higher local summability condition on $a$ would make the strong maximum principle applicable, from which we would have strict positivity of the solutions in the above existence theorems, see e.g.  \cite{MR2356201}.
\end{rem}

We achieve the above results seeking bound state solutions, by means of a suitable topological argument. Here a new decomposition Theorem \ref{globalcomp0} for PS-sequences associated with $F_a$ plays a role  
 in the spirit of \cite{CeSoSt86JFA,St84MZ}, which holds under our assumption \eqref{b}. Roughly speaking this is a consequence of combining our new Theorem \ref{TNE} and recent classification results 
 for the equation 
 \begin{equation*}
     -\Delta_p U=|U|^{p^*-2}U\quad \textrm{on}\,\,\R^N
 \end{equation*}
 see \cite{FaMeWi19CVPDE,Sc16AM,Ve16JDE}, with ideas from \cite{CM19ESAIM} where Cerami and the second author have dealt with the case $p=2$ and constant $a_\infty,$ as well as with \cite{MeWi10DCDS} where the first author and Willem have dealt with case $p\neq 2$ and $a_\infty\equiv 0.$ This analysis combined with Theorem \ref{TNE} and Proposition \ref {eqmin}  allows us to recover the usual energy range for compactness (values between the first two {\it quantised} energy levels, see Corollary \ref{comp}), which is then used in a subsequent topological argument.  
In fact, by degree theory we can then identify some changes in the topology of the sub-levels of the action functional, and therefore detect the existence of nontrivial solutions in correspondence of certain critical levels.
Here we use notions of `barycenter' and `concentration rate' for functions $u\in S_{p^*}$, which differ from those introduced in \cite{BeCe90JFA} (see Section 4). 
Namely, we use the definition of barycenter introduced in \cite{CePa03CVPDE} for subcritical nonlinearities together with an exponential kernel to ease certain convergence arguments involving the concentration rate of functions defined on the whole of $\R^N$. 
As a byproduct of this choice, in the topological arguments a notion of barycenter of functions is used without projections on the unit sphere of $\R^N,$ unlike in the related aforementioned works where similar tools have been used.    \\ 
\noindent Our approach to existence is mainly based on the observation that the topology of the sublevels of the functional $F_a$ we deal with in the case of Theorem \ref{R^N} $(\Omega=\R^N),$ is preserved for sufficiently large bounded domains. 
This approach is in some sense opposite to that used, for instance, in \cite{CeDeSo05CVPDE}, where it is shown that the existence of solutions proven on bounded domains is preserved in the `limit' to $\R^N$. It is worth observing that by our approach and a scaling argument, some of the main results in  \cite{Pa96AIHP,BrWi08JFA,MeWi10DCDS}
can be recovered and somewhat improved, allowing $a$ to have a more general asymptotic behavior at infinity.

\begin{rem}
By the proof of Theorem $\ref{R^N}$ we understand that as $\|a_\infty\|_\infty\to 0$  the low energy solution $u_l$ `disappears' whereas the high energy solution $u_h$ converge to the corresponding one for $a_\infty\equiv 0$.

\end{rem}

\begin{rem}
    We point out that in Theorem \ref{bounded} the low energy solution $u_l$ exists also if $a_\infty\equiv 0$, contrary to the case $\Omega=\R^N$.   
    This different behaviour is related to the effect of the boundary of $\Omega$ on the action functional.
In this framework, $u_l$ disappears as  $\|a_\infty\|_\infty\to 0$  and $\rho(\Omega)\to \infty$, where 
\beq
\label{rho}
\rho(\Omega)=\max\{r>0\ :\ B_r(0)\subseteq\overline\Omega\}.
\eeq

    In general, the solution $u_l$ `disappears' when $\alpha\equiv 0$, also if $a_\infty>0$, as a consequence of Theorem \ref{TNE}.
\end{rem}

This paper is organised as follows. In Section \ref{nonexistencesection} we prove the main nonexistence result Theorem \ref{TNE}, which is then used in Section \ref{compactnesssection} to identify a suitable energy range for a local Palais-Smale condition to hold. Section \ref{Bahricenters} is devoted to some key energy estimates and auxiliary results which are obtained reviewing tools such as the barycenter and concentration rate of a function; finally in Section \ref{unboundedsection} and Section \ref{boundedsection} we prove our main existence and multiplicity results.

$\phantom .$

{\bf { Notation}}

\begin{itemize}

\item
$\|\cdot\|$ denotes the norm in $X$, that is
$$
\|u\|=\left(\int_{\Omega}(|\D u|^p+a |u|^p)dx \right)^{1/p}; 
$$

\item 
either $\|u\|_q$ or $|u|_q$, for $1\le q\le +\infty,$ may be used to denote the norm in the Lebesgue space
$L^q(\R^N);$ the norm of $u$ in $L^q(\Omega)$,
$\Omega\subset\R^N$, may be denoted by $|u|_{q,\Omega}$;

\item
$B_\rho(y)$, $\forall y\in\R^3$, denotes the open ball of radius
$\rho$ centered at $y$; 

\item 
any $u\in W^{1,p}_0(\Omega)$ will be considered as $u\in   W^{1,p}(\R^N)$ , setting $u=0$ in $\R^N\setminus\Omega$.

\end{itemize}

$\phantom .$

{\bf { Acknowledgements.}} The authors are partially supported by GNAMPA 2023 research project ``Problemi ellittici non-lineari con $p$-Laplaciano e mancanza di compattezza''. C.M. would like to warmly thank the Department of Mathematics at Universit\`a degli Studi di Roma - Tor Vergata, as well as Collegio San Carlo (Modena), for the kind hospitality when this paper was being written. R.M. is partially supported also by the MIUR Excellence Department Project MatMod@TOV  
awarded to the Department of Mathematics, University of Rome Tor Vergata, CUP E83C23000330006.



\sezione{Proof of the nonexistence result Theorem \ref{TNE}} \label{nonexistencesection}

\begin{proof}[Proof of Theorem \ref{TNE}]
For sake of clarity we break the proof into several steps. \newline

{\it Step 1.} We claim that under the hypotheses of the theorem, every weak solution $u$ satisfies the Pohozaev identity

\beq
\label{Pohozaev}
\frac{N-p}{p}\int_{\R^N}|\nabla u|^p=N\int_{\R^N}\Big(\frac{|u|^{p^*}}{p^*}-a_\infty\frac{|u|^p}{p}\Big).
\eeq
We observe that this is enough to conclude, as in combination with Nehari's classical identity
\begin{equation*}
    \int_{\R^N}|\nabla u|^p=\int_{\R^N}\Big(|u|^{p^*}-a_\infty|u|^p\Big)
\end{equation*}
we find that every weak solution satisfies necessarily
\[\Big(\frac{p^*}{p}-1\Big)\int_{\R^N}a_\infty |u|^p=0,\] which yields $u\equiv 0,$ as $a_\infty$ is positive and bounded away from zero.\newline

{\it Step 2.} Since $a_\infty$ is a purely angular function, it follows
 for every $x\in \R^N$  that
\begin{equation}\label{nullscalar}
    x\cdot\nabla a_\infty=0.
\end{equation}

{\it Step 3.} In order to justify \eqref{Pohozaev}, we perform a classical integration by parts argument, after checking we have enough regularity for it. To this aim we observe that, adapting Moser's iteration in \cite[Appendix E]{Peral97Notes}, see also \cite{MR0240748}, it follows that $u\in L^{\infty}_{\textrm{loc}}(\R^N).$ This implies, by  \cite{Db83NA} that $u\in C^{1,\alpha}_{\textrm{loc}}(\R^N).$ By \cite{MR0727034}, when $p\leq 2$ it also holds that $u\in W^{2,p}_{\textrm{loc}}(\R^N).$ With these preliminary remarks in place for $p\leq 2$ we can  
multiply the equation 

\begin{equation*}
    -\Delta_p u+a_\infty(x)|u|^{p-2}u=|u|^{p^*-2}u \qquad \mbox{ in }\R^N
\end{equation*}

by $x_i\partial_i u, \, i=1,...,N,$ and integrate over a generic ball of radius $R,$ say $B_R,$ with outer unit vector denoted by $n(\cdot),$ obtaining by the divergence theorem
that

\begin{align*}
    \int_{B_R} \Delta_p u \,x_i\partial_i u(x)dx&=\int_{\partial B_R}|\nabla u(\sigma)|^{p-2} \partial_i u (\sigma)\sigma_i \nabla u \cdot n \,d\sigma\\ &\qquad \qquad -\int_{B_R} |\nabla u(x)|^{p-2}\nabla u(x)\cdot\nabla[x_i\partial_i u(x)]dx,
\end{align*}
which is the only step involving second (weak) derivatives. Again, by the divergence theorem, the last integral can be written as 

\begin{align*}
 \int_{B_R} |\nabla u(x)|^{p-2}\nabla u(x)\cdot\nabla[x_i\partial_i u(x)]dx & = \int_{B_R} |\nabla u(x)|^{p-2}|\partial_i u(x)|^2 dx \\
 & + \frac{1}{p}\int_{\partial B_R}|\nabla u(\sigma)|^{p} \sigma_i n_i d\sigma-\frac{1}{p}\int_{B_R} |\nabla u(x)|^{p}dx .
\end{align*}
On the other hand, we also have 
 
 \begin{align*}
     & \int_{B_R} \Big(|u|^{p^*-2}u-a_{\infty}|u|^{p-2}u\Big)x_i\partial_i u(x)dx\\& =-\int_{B_R} \frac{|u|^{p^*}}{p^*}dx+\int_{\partial B_R} \frac{|u|^{p^*}}{p^*}\sigma_i n_i d\sigma - \int_{B_R}a_{\infty}|u|^{p-2}ux_i\partial_i u.
 \end{align*}
 In particular, the last integral can be written as 
\begin{align*}
     &\int_{B_R}a_{\infty}|u|^{p-2}u x_i\partial_i u ~dx\\ 
     & =\int_{\partial B_R}a_\infty \frac{|u|^p}{p}\sigma_i n_i d\sigma-\int_{B_R}\Big(a_\infty \frac{|u|^p}{p}+\frac{|u|^p}{p}x_i\partial_i a_\infty\Big)dx.
 \end{align*}
 Putting the calculations above together, summing up on $i,$ we have  
 
 \begin{align}\label{localPoho}
 & N\int_{B_R} \Big(\frac{|u|^{p^*}}{p^*}-a_\infty \frac{|u|^p}{p}\Big)dx-\int_{B_R}\frac{|u|^p}{p}x\cdot \nabla a_\infty dx+\Big(1-\frac{N}{p}\Big)\int_{B_R} |\nabla u(x)|^{p}dx\\ \nonumber
&\qquad \qquad =\int_{\partial B_R} |\nabla u(\sigma)|^{p-2}\nabla u\cdot \sigma\,\nabla u\cdot n  \,d\sigma-\frac{1}{p}\int_{\partial B_R} |\nabla u(\sigma)|^{p}\sigma\cdot n  \,d\sigma\\\nonumber &\qquad +\int_{\partial B_R}  \Big(\frac{|u|^{p^*}}{p^*}-a_\infty \frac{|u|^p}{p}\Big)\,\sigma\cdot n d\sigma.    \end{align}
We note now that this identity immediately yields \eqref{Pohozaev}. In fact, in the left hand side the integral involving $x\cdot \nabla a_\infty$ is zero by \eqref{nullscalar} in Step 2, whereas the right hand side is bounded by 
$$M(R)=\Big(1+
\frac{1}{p}\Big)R \int_{\partial B_R} |\nabla u(\sigma)|^{p}  \,d\sigma + R\int_{\partial B_R} \Big(\frac{|u|^{p^*}}{p^*}+a_\infty \frac{|u|^p}{p}\Big)d\sigma.$$ Taking into account the summability of the integrands, it follows that there exists a sequence $R_k \rightarrow \infty$ such that $M(R_k)\rightarrow 0,$ which yields \eqref{Pohozaev} by e.g. the monotone convergence theorem.\\
For $p> 2,$ a regularisation argument allows to obtain the same local Pohozaev identity \eqref{localPoho} as a limit of an analogous one for a sequence of $C^2$-approximated solutions, as in \cite[Proposition 3.1]{AlMeMo20AMPA}; see also \cite[Lemma 2.1]{FaMeWi19CVPDE}. This concludes the proof.
\end{proof}



\sezione{Compactness}\label{compactnesssection}
We start this section by briefly justifying the nonexistence of groundstates, as well as by introducing some notation for Sobolev's optimisers.

\medskip

\noindent {\em Proof of Proposition \ref{NEM}}.\quad
It is well known that $S$ is attained in $\cD^{1,p}(\R^N)$ by the functions
$$
\bar U_{y,\e}:= \frac{U_{y,\e}}{|U_{y,\e}|_{p^*}},
$$
where
\begin{equation}\label{Talenti}\tag{$U_{y,\e}$}
U_{y,\e}=\frac{c(N,p)\,\varepsilon^{\frac{N-p}{p(p-1)}}}{\Big(\varepsilon^{\frac{p}{p-1}}+|\cdot-y|^{\frac{p}{p-1}}\Big)^{\frac{N-p}{p}}},\quad (y,\varepsilon)\in\R^N\times\R_+,
\end{equation}
are the only positive solutions in $\mathcal D^{1,p}(\R^N)$ to 
\[-\Delta_p U=|U|^{p^*-2}U\quad \textrm{on}\,\,\R^N\nonumber\]
and where $c(N,p)>0$ is chosen so that
\begin{align*}
  |DU_{y,\e}|^p_p =S^{N/p}\\
  |U_{y,\e}|_{{p^*}}^{p^*}=S^{N/p}
\end{align*}
(see \cite{GuVe88JDE,FaMeWi19CVPDE,Sc16AM,Ve16JDE,Ta76AMPA}).
Then, it is standard to see that
$$
F_a(\phi_\e)\to S\qquad\mbox{ as }\e\to \infty,
$$
where
$$
\phi_\e(x):=\frac{c(x) U_{0,\e}(x)}{|c(x)U_{0,\e}(x)|_{p^*}}\qquad c\in\cC^\infty_0(B_1(0))\mbox{ a cut-off function}
$$
(see \cite{BrNi83CPAM} and \cite[pp.~892,~893]{GuVe89NA}).
So, \eqref{eqmin} is proved. 
To see that the infimum in \eqref{eqmin} is not attained, assume by contradiction that it is reached by a function $v\in S_{p^*}$. 
Then 
$$
S\le F_0(v)\le F_a(v)=S
$$
implies that $F_0(v)=S$, so that $v=\bar U_{y,\e}$ for suitable $y\in\R^N$ and $\e>0$.
Since $F_a(\bar U_{y,\e})=S$ it follows that $\Omega=\R^N$ and $a(x)\equiv 0$, which is a contradiction by \eqref{b}.
\qed

\medskip
Before dealing with the relevant compactness properties of $F_a$ we formulate them in terms of a `free functional' $J_a$ making our statements slightly more general than the way these will be applied later. \\ \noindent We recall that a functional $J\in C^1(X;\R)$ defined on some Banach space $X$ satisfies the Palais-Smale condition at the level $c,$ $(PS)_c$- condition in short, if every sequence $(u_n) \subset X$ such that 
\begin{equation} 
\label{PSseq} 
J(u_n)\rightarrow c\quad \textrm{and} \quad J'(u_n)\rightarrow 0, \qquad \textrm{as} \quad n\rightarrow \infty,\end{equation}
possesses a strongly convergent subsequence. 
A sequence $(u_n) \subset X$ such that \eqref{PSseq} holds is called Palais-Smale sequence for $J$ at the level $c$, or $(PS)_c$-sequence. 

Similar definitions hold for a regular functional $F$ constrained on a regular manifold of codimension one, see e.g. \cite{MR2292344}.

 For any $u\in X$ we set
\[J_a(u)=\int_{\R^N}\Big(\frac{|D u |^p+a(x)|u|^p}{p}-\frac{|u|^{p^*}}{p^*}\Big)\nonumber\]
and for all $u\in \cD^{1,p}(\R^N)$ 
\[J_0(u)=\int_{\R^N}\Big(\frac{|D u |^p}{p}-\frac{|u|^{p^*}}{p^*}\Big).\nonumber\]

\noindent Set 
\[\mathcal N=\{u\in X\setminus\{0\}\ :\ (J_a'(u),u)=0\}\nonumber.\]

\noindent The following lemma states that the universal constant $S^{N/p}{N}$ is a strict lower bound for $J_a|_{\mathcal N}$ and in particular for $J_a$ restricted on the set of its nontrivial critical points. 

\begin{lemma}[$S^{N/p}{N}$ is a strict lower bound for $J_a|_{\mathcal N}$] \label{lowerboundm} For any $p\in(1,N)$ set 
\[m=\inf\{J_a(u)\ :\ u\in\mathcal N\}.\] Then,
it holds that \[m=\frac{S^{N/p}}{N}.\]    Moreover,  $m$ is not attained.
\end{lemma}

\begin{proof} By a classical homeomorphism argument, see e.g. \cite[Theorem 4.2]{MR1400007}, this lemma is equivalent to Proposition \ref{NEM}, which is justified above. This concludes the proof.
\end{proof}

\begin{teo}[Energy quantisation for critical points of $J_a$]
\label{globalcomp0}

 Let $(u_n) \subset X$ be a $(PS)_c$ sequence for $J_a.$ Then, passing if necessary to a subsequence, $u_n\rightharpoonup u,$ where $u$ is a possibly nontrivial solution to
$$-\Delta_p u +a(x)|u|^{p-2}u=|u|^{p^*-2}u\quad \textrm{in}\quad X,$$
and exist $k\in \mathbb N \cup \{0\}$ nontrivial functions, denoted by $v_1,...,v_k\in \mathcal D^{1,p}(\R^N)$ when $k\geq 1,$ satisfying
\begin{equation}\label{BUBBLE}
  -\Delta_p v_i =  |v_i|^{p^*-2}v_i\quad \textrm{in}\quad \R^N, \qquad i=1,...,k  
\end{equation}

such that for some $k$ sequences of points $(y^i_n) \subset \R^N$ and radii $(\varepsilon^i_n) \subset \R_+,$  it holds that

$$
\|u_n-u-\sum^k_{i=1}(\varepsilon^i_n)^{(p-N)/p}v_i ((\cdot-y^i_n)/\varepsilon^i_n)\|_{\mathcal D^{1,p}(\R^N)}\rightarrow 0, \quad n\rightarrow \infty,
$$
 
$$
\|u_n\|^p\rightarrow \|u\|^p+\sum^k_{i=1}
\|v_i\|^p_{\mathcal D^{1,p}(\R^N)}, \quad n\rightarrow \infty,
$$
and
\begin{equation}
\label{quantum}
J_a(u)+\sum^k_{i=1}J_0 (v_i)=c.
\end{equation}
\end{teo}
\begin{proof}
The statement when $a_\infty\equiv 0$ is the case originally studied by Benci and Cerami \cite{BeCe90JFA} for $p=2,$ and more recently by Alves \cite{MR1926623}, Mercuri and Willem \cite{MeWi10DCDS}, see also Farina et al.  \cite{FaMeWi19CVPDE} for $p\neq 2,$ where in this nonlinear case, decomposition results for the energy functional and its derivative along PS-sequences in the spirit of Br\'ezis and Lieb \cite{MR0699419} have been obtained. 
\newline Decomposition results of the same flavor and with $a_\infty>0$ constant, have been obtained recently in different elliptic contexts by Alves et al. \cite{AFM21DCDS}, Cerami and Molle \cite{CM19ESAIM}, and Guo et al. \cite{GLLM22arXiv}, where it is shown how to justify the splitting in the norms corresponding to different function spaces ($X$ and $\mathcal D^{1,p}$ in our case) with the use of appropriate cut-off functions. \newline
From these preliminaries in place, we may follow e.g. the scheme of the proof given in \cite[Theorem 8.13]{MR1400007} and \cite[Theorem 1.2]{MeWi10DCDS}, which are in the spirit of Struwe's global compactness result \cite{St84MZ}, Br\'ezis and Coron \cite{MR0784102}, and the limit-case of the concentration-compactness principle of P.-L. Lions \cite{MR0834360, MR0850686}. \newline
The justification of the finiteness of the number $k$ of `bubbles' $v_i,$ rely as usual on a uniform lower bound for the restriction of $J_a$ on the set of its nontrivial critical points, which is provided by $m$ in Lemma \ref{lowerboundm} above. \newline
Finally, the fact that `standard bubbles' $v_i \in \mathcal D^{1,p}(\R^N)$ solutions to \eqref{BUBBLE} are the only responsible for a possible loss of strong convergence for these sequences in $X,$ can be explained by a blow-up argument using Theorem \ref{TNE} within the stage of the proofs in \cite[Theorem 8.13]{MR1400007} and \cite[Theorem 1.2]{MeWi10DCDS} where nonexistence results are used for the limiting PDE problems `at infinity'. \newline In fact, let $a_\infty\not\equiv 0.$
Since the sequence $(u_n)$ is bounded in $X$, always passing if necessary to suitable subsequences in what follows, we can assume that $u_n\rightharpoonup u$ in $X,$ and $u_n\rightarrow u$ as well as $\nabla u_n \rightarrow \nabla u$ almost everywhere, see \cite{MeWi10DCDS,MR2560133}.  One can see that $J_a'(u)=0$ and $u^1_n:=u_n-u$ is such that
\begin{align*}
    &\|u^1_n\|^p=\|u_n\|^p-\|u\|^p+o(1),\\
    & J_{a_\infty}(u^1_n)\rightarrow c-J_a(u),\\
    & J'_{a_\infty}(u^1_n)\rightarrow 0 \quad \textrm{in}\,\, W^{-1,p'}.
\end{align*}
We can assume that
$$\int_{\R^N}|u^1_n|^{p^*}dx>\delta$$ for some $\delta>0.$ Using well-known properties of the L\'evy concentration function $$Q_n(r)=\sup_{y\in \R^N}\int_{B(y,r)}|u^1_n|^{p^*}dx, $$ there exists a sequence $(\varepsilon^1_n) \subset ]0,\infty[$ and   a sequence $(y^1_n) \subset \R^N$ such that
\begin{equation*}\label{LLL}
\delta=\sup_{y\in \R^N}\int_{B(y,\varepsilon^1_n)}|u^1_n|^{p^*}dx=\int_{B(y^1_n,\varepsilon^1_n)}|u^1_n|^{p^*}dx .
\end{equation*}
We set $v^1_n(x)=(\varepsilon^1_n)^{(N-p)/p}u^1_n(\varepsilon^1_n x+y^1_n).$ We can assume that $v^1_n\rightharpoonup v_1$ in $\mathcal D^{1,p}(\R^N)$ and $v^1_n\rightarrow v_1$ a.e. on $\R^N,$ with $v_1\not\equiv 0$ (see e.g. \cite[p. 480]{MeWi10DCDS}). Note that $v^1_n$ solves
\begin{equation*}
-\Delta_p v^1_n + (\varepsilon^1_n)^p\, \Theta\Big(\frac{\varepsilon^1 _n x+ y_n}{|\varepsilon^1 _n x+ y_n|}\Big)|v^1_n|^{p-2}v^1_n=|v^1_n|^{p^*-2}v^1_n+o(1)\qquad \textrm{in}\,\,W^{-1,p'}.
    \end{equation*}
The limiting behavior of $y^1_n$ and $\varepsilon^1_n$ can then be understood ruling out all possible cases, except $\varepsilon^1_n \rightarrow 0.$ For instance if $\varepsilon^1_n \rightarrow \varepsilon>0, \,\, y^1_n\rightarrow y \in \R^N,$ it follows by the continuity of $a_\infty$ that $v_1$ would solve
\begin{equation*}
-\Delta_p v_1 + \varepsilon ^p\, \Theta\Big(\frac{\varepsilon  x+ y}{|\varepsilon  x+ y|}\Big)|v_1|^{p-2}v_1=|v_1|^{p^*-2}v_1,
    \end{equation*}
yielding a contradiction with Theorem \ref{TNE}.
\end{proof}

\begin{teo}[Compactness]
Let $p\in (1,N)$ and $c\in (S^{N/p}/N,2S^{N/p}/N).$ Then $J_a$ satisfies the $(PS)_c$ condition.   
\end{teo}
\begin{proof}
For any $(PS)_c$ sequence for $J_a$, by \eqref{quantum} in Theorem \ref{globalcomp0}
we have for some $k=\mathbb N\cup \{0\}$ that
\[c=J_a(u)+k\frac{S^{N/p}}{N}\]
and where $J_a'(u)=0.$ Here we have taken into account that $c\in (S^{N/p}/N,2S^{N/p}/N)$ necessarily implies that $(v_i)_{i=1,...,k}$ if they are nontrivial, by a standard argument (see also Proposition \ref{Psegno}) cannot change sign and therefore they coincide, by  recent classification results \cite{FaMeWi19CVPDE,Sc16AM,Ve16JDE}), up to the sign, to some of the functions in \eqref{Talenti}, whose energy is $S^{N/p}/N$ by \cite{Ta76AMPA}).  The case $u\equiv 0$ cannot occur as the above would immediately give a contradiction with the range we consider for $c$. If $u\not\equiv 0,$ Lemma \ref{lowerboundm} gives that the only possibility compatible with the assumption on $c$ is $k=0,$ yielding the conclusion again by Theorem \ref{globalcomp0}. 
\end{proof}

It is standard to see that the above theorems imply the following `constrained' counterparts involving $F_a$.


\begin{cor}[Minimizing sequence for $F_a$ on $S_{p^*}$] 
\label{globalcomp}
Let $a$ satisfy assumption \eqref{b} and let $(u_n) \subset X$ be a $(PS)_S$ sequence for $F_a$  constrained on $S_{p^*}$ (see \eqref{Sobolev}).
Then
$$
u_n=U_{ y_n,\e_n}+\Phi_n,\quad \ \mbox{ with }\ \Phi_n\to 0\ \mbox{ in }\cD^{1,p}(\R^N)
$$
for a sequence of points $(y_n) \subset \R^N$ and a sequence of radii $(\varepsilon_n) \subset \R_+,$.
\end{cor}

\begin{cor}[Compactness for $F_a$]
\label{comp}
Let $p\in (1,N)$ and $c\in (S ,2^{p/N}S)$. 
Then, $F_a$ constrained on $S_{p^*}$ satisfies the $(PS)_c$ condition.   
\end{cor}

\sezione{Barycenter, concentration rate, and key energy estimates}\label{Bahricenters}

In the following, we are considering $\Omega=\R^N$ and $a\not\equiv 0$.

We start by introducing the notion of \textit{barycenter} and \textit{concentration rate} of a function 
$u \in \cD^{1,p}(\R^N)\setminus \{0\}$.
This definition of barycenter  has been introduced in the subcritical case in  \cite{CePa03CVPDE} and here we pair it with a measure of the rate of concentration.
 
 Let 
\beq\label{nu}
\nu(u)(x)= \int_{B_1(x)}|u(y)|^{p^*}dy,
\eeq
we observe that  $\nu(u)$ is bounded, continuous and vanishing at infinity, so  the function
\beq\label{hat}
\widehat{u}(x)=\left[\nu(u)(x)-\frac{1}{2}\max \nu(u)\right]^{+}
\eeq
is well defined, continuous, and has compact support. 
Then, we introduce
$\beta: \cD^{1,p}(\R^N)\setminus\{0\}\to \R^N$ and $\gamma:\cD^{1,p}(\R^N)\setminus\{0\}\to \R $
setting
\beq
\tag{$\beta$}\label{beta}
\beta(u)=\frac{1}{|\widehat{u}|_{1}}\int_{\R^N}\widehat{u} (x)\, x\, dx,
\eeq
\beq
\tag{$\gamma$}\label{gamma}
\gamma(u)=\frac{1}{| {u}|_{p^*}^{p^*}}\int_{\R^N} |u (x)|^{p^*} \, e^{-|x-\beta(u)|} \, dx.
\eeq

The maps $\beta$ and $\gamma$ are well defined, because $\widehat{u}$ has compact support, and it is not difficult to verify that they enjoys the following properties:
\begin{itemize}
\item $\beta,\gamma$ are continuous in $\cD^{1,p}(\R^N)\setminus \{0\}$;
\item if $u$ is a radial function, then $\beta(u)=0$;
\item for $z\in \R^N$ and $u\in \cD^{1,p}(\R^N)\setminus\{0\}$ it holds that
\beq
\label{invarianza}
\beta(u(x-z))=\beta(u)+z,\qquad \gamma(u(x-z))=\gamma(u) . 
\eeq
\end{itemize}
 
\begin{df}
\label{def_riscalate}
Let $u\in  \cD^{1,p}(\R^N)$ and $\e>0$. 
We set
\beq
\label{mus}
u_{y,\e}(x):=\e^{-\frac{N-p}{p}}\, u\left(\tfrac {x-y}{\e}\right).
\eeq
\end{df}

\begin{rem}
\label{BG}
Let $u\in  \cD^{1,p}(\R^N)\setminus\{0\}$ be such that $\beta(u)=0$. 
 From \eqref{invarianza} and the dominated convergence theorem, we easily see that
$$
\lim\limits_{\e\to 0}\gamma(u_{0,\e})=1,\qquad \lim\limits_{\e\to\infty}\gamma(u_{0,\e})=0.
$$
\end{rem}

\begin{lemma}
\label{L5.3}
Let $\alpha,a_\infty$ as in \eqref{b}. 
If $a_\infty+\alpha\not\equiv 0$,  then 
$$
\cB(a_\infty,\alpha):=\inf\{F_a(u)\ :\ u\in S_{p^*}(\R^N),\ \beta(u)=0,\ \gamma(u)=\tfrac 12 \}>S.
$$
\end{lemma}

\proof
  
Inequality $\cB(a_\infty,\alpha)\ge S$ follows immediately from \eqref{Sobolev} and $a\ge 0$.
Then, assume by contradiction that $\cB(a_\infty,\alpha)=S$ holds.
Let $(u_n)$ be a minimizing sequence for $\cB(a_\infty,\alpha)$.
Observe that 
\beq
\label{este}
S\le F_0(u_n)\le F_a(u_n)\to S,
\eeq
so, in particular, $(u_n)$ is a minimizing sequence for $F_0$, too.
By the Ekeland variational principle, we can assume that $(u_n+\phi_n)$, with $\phi_n\to 0$ in $\cD^{1,p}(\R^N)$, is a Palais-Smale sequence for $F_0$.
Then, by Corollary \ref{globalcomp} applied to the case $a\equiv 0$, we can assume that there exists a sequence of points $(y_n)$ and a sequence of positive values $(\e_n)$ such that
\beq
\label{1855}
u_n(x)=U_{ y_n,\e_n}(x)+\Phi_n(x),\quad \ \mbox{ with }\ \Phi_n\to 0\ \mbox{ in }\cD^{1,p}(\R^N).
\eeq

Now, we claim that, up to a subsequence, 
\beq
\label{1852}
(a)\ \lim_{n\to\infty}\e_n=\bar\e>0,\qquad (b)\ \lim_{n\to\infty}y_n=0.
\eeq
In order to prove \eqref{1852}{\em (a)}, we start to show that $(\e_n)$ is bounded.

\smallskip

Assume, by contradiction, that $\e_n\to\infty$, up to a subsequence. 
Then, by  \eqref{1855} $u_n\to 0$ in $L^{p^*}_{\loc}(\R^N)$ so, for every $\eta, r>0$ we have $|u_n|^{p^*}_{p^*,B_r(0)}<\frac{\eta}{|B_r(0)|}$ for large $n$.
Hence $\gamma(u_n)\le \eta+e^{-r}$ and $\gamma(u_n)\to 0$ as $\eta$ and $r$ were arbitrary, contradicting that $\gamma(u_n)\equiv\tfrac 12$.

\smallskip

Assume now, contradicting our claim, that $\e_n\to0$, up to a subsequence. 
In such a case, 
\beq
\label{new}
\nu(u_n)(y_n)\to 1\ \mbox{ and }\nu(u_n)(x)\to 0 \mbox{ uniformly in }\R^N\setminus B_r(y_n),\ \forall r>0.
\eeq
As a consequence, for every $r>0$ we have $\supp \widehat u\subseteq B_r(y_n)$ for large $n$, that implies $\beta(u_n)\in B_r(y_n)$. 
Since As $r$ was arbitrary, we deduce $y_n=\beta(u_n)+o(1)=o(1)$.
So,
\beq
\label{new1}
\begin{split}
\gamma(u_n)&=\int_{\R^N}|u_n(x)^{p^*}e^{-|x|} dx\ge \int_{B_{\log \frac43}(0)}|U_{y_n,\e_n}(x)+\Phi_n(x)|^{p^*}e^{-|x|}dx\\
& \ge \frac34\int_{B_{\log \frac43}(0)}|U_{ y_n,\e_n}(x)+\Phi_n(x)|^{p^*}dx=\frac34(1+o(1)),
\end{split}
\eeq
  contrary to $\gamma(u_n)\equiv \tfrac12$.
Then, the proof of \eqref{1852}$(a)$ is completed.

\smallskip

In order to prove \eqref{1852}$(b)$, let us remark that $ \e_n\to \bar\e>0$ and $\Phi_n\to 0$ in $L^{p^*}(\R^N)$ imply $\widehat u_n(x)-\widehat U_{y_n,\e_n}\to 0$ in $B_{2r}(y_n)$ and $\widehat u_n(x)-\widehat U_{ y_n,\e_n}\equiv 0$ in $\R^N\setminus B_{2r}(y_n)$, where $r>0$ is such that $\supp \widehat U_{0,\bar \e}\subseteq B_r(0)$.

Then $\beta(u_n)-\beta (U_{y_n,\e_n})\to 0$, that is $y_n\to 0$ because $\beta(u_n)\equiv 0$ and $\beta (U_{y_n,\e_n})= y_n$, $\forall n\in n$.
So, also \eqref{1852}$(b)$ is proved.

\smallskip

By \eqref{1852} and \eqref{b} we get 
\beq
\label{Clapp}
\lim_{n\to\infty} F_a(u_n)=S+\int_{\R^N}a (x)U_{0,\bar\e}^pdx>S
\eeq
which is not compatible with \eqref{este}.
\qed

\medskip

In the following, we consider $\phi\in\cC^\infty_0(\R^N)$ such that 
\beq
\label{BP}
\left\{
\begin{aligned}
&\phi\quad\mbox{ is radial }\\
&\supp \phi\subset B_1(0)\\
&F_0(\phi)<\min\{\cB(0,\alpha),2^{p/N}\}.
\end{aligned}
\right.
\eeq
Observe that this choice of $\phi$ depends on $\alpha\in L^{N/p}(\R^N)\setminus\{0\}$.

\begin{lemma}
\label{lemma}
Let $\phi_{z,\e}$ be according to Definition \ref{def_riscalate}, then
\begin{itemize}
\item[(i)] $\beta(\phi_{z,\e})=z$, \qquad $\forall z\in \R^N$;
\item[(ii)] $\lim\limits_{\e\to 0}\gamma (\phi_{z,\e} )=1$, $\lim\limits_{\e\to\infty}\gamma(\phi_{z,\e} )=0$, uniformly with respect to $z\in\R^N$.
\end{itemize}
\end{lemma}

The proof of point $(i)$ follows direcly from the radial symmetry of $\phi$ and of the definition of $\beta$.
 Point $(ii)$  follows from Remark \ref{BG}, taking into account \eqref{invarianza}.

\begin{lemma}
\label{Lcet}
Let us assume $\alpha\in L^{N/p}(\R^N)$, then
\beq
\label{cet1}
(i)\ \lim_{\e\to\infty}\int_{\R^N}\alpha(x)\phi_{z,\e}^{p}dx=0,\quad (ii)\   \lim_{\e\to 0}\int_{\R^N}\alpha(x)\phi_{z,\e}^{p}dx=0,\quad \mbox{uniformly w.r.t.} \ z\in\R^N.
\eeq
and 
\beq
\label{cet2}
 \lim_{|z|\to\infty}\int_{\R^N}\alpha(x)\phi_{z,\e}^{p}dx=0,\quad \mbox{uniformly w.r.t.}\ \e>0.
\eeq
 \end{lemma}
By Lemma \ref{Lcet} we have that
\beq
\label{prova}
\lim_{R\to\infty}\max\{F_\alpha(\phi_{z,\e})\ :\ (z,\e)\in \partial Q_R\}=F_0(\phi),
\eeq
where
\beq
\label{Far}
Q_R:= B_R(0)\times(R^{-1},R),\qquad R>1.
\eeq
 
\begin{prop}\label{Pcet}
Let us fix $\alpha\in L^{N/p}(\R^N)\setminus\{0\}$ and let $a_\infty$ be as in \eqref{b}. Then, $\bar R=\bar R(\alpha)>1$ existes such that 
$$
\cG(a_\infty):=\max\{F(\phi_{z,\e})\ :\ (z,\e)\in \partial Q_{\bar R}\} 
$$
verifies
\beq
\label{aa1}
\cG(0)<\min\{\cB(0,\alpha),2^{p/N}S\}.
\eeq
Moreover,
\beq
\label{aa2}
\lim_{|a_\infty|_\infty\to 0}\cG(a_\infty)=\cG(0)
\eeq
and
\beq
\label{Pac}
\gamma(\phi_{z,1/\bar R})>1/2,\qquad \gamma(\phi_{z,\bar R})<1/2\qquad \forall z\in\R^N.
\eeq
\end{prop}

\proof
Estimate \eqref{aa1} follows by \eqref{BP} and \eqref{prova}.
The concentration rate in \eqref{Pac} can be deduced by Lemma \ref{lemma}. 
The asymptotic behavior \eqref{aa2} follows as $\max_{(z,\e)\in\partial Q_{\bar R}}|\phi_{z,\e}|_1<\infty$.
\qed

\medskip

\noindent {\em Proof of Lemma \ref{Lcet}.}\ 
First, let us prove \eqref{cet1}{\em (i)}.

Fix $\eta>0$, and pick $r_\eta>0$ such that $|\alpha|_{N/p,\R^N\setminus B_{r_\eta}(0)}<\tfrac{\eta}{ 2}$.
Now, let $\e_\eta>0$ be such that $\max\phi_{z,\e}<\tfrac\eta2(|\alpha|_{1,  B_{r_\eta}(0)})^{-1}$, $\forall\e>\e_\eta$.
Then, for $\e>\e_\eta$,
$$
\int_{\R^N}\alpha(x)\phi_{z,\e}^{p}dx=\int_{\R^N\setminus B_{r_\eta}(0)}+\int_{  B_{r_\eta}(0)}\alpha(x)\phi_{z,\e}^{p}dx\le|\alpha|_{N/p,\R^N\setminus B_{r_\eta}(0)}+\max\phi_{z,\e}|\alpha_{1,  B_{r_\eta}(0)}<\eta.
$$

\smallskip

To prove \eqref{cet1}{\em (ii)}, it is enough to observe that
$$
\int_{\R^N}\alpha(x)\phi_{z,\e}^{p}dx\le |\alpha|_{N/p,\supp \phi_{z,\e}}
$$
and that $|\supp \phi_{z,\e}|\to 0$ as $\e\to 0$, uniformly w.r.t. $z\in \R^N$.

\smallskip

Now, let us prove \eqref{cet2}.

 For any $\eta>0$, by \eqref{cet1}{\em (i)} there exists $\e_\eta>0$ such that 
$$
\int_{\R^N}\alpha(x)\phi_{z,\e}^{p}dx<\eta \qquad\forall \e>\e_\eta,\quad
\forall z\in\R^N.
$$
Then we are left to prove that there exists $R_\eta>0$ such that 
\beq
\label{Sorrentino}
\int_{\R^N}\alpha(x)\phi_{z,\e}^{p}dx<\eta \qquad \forall \e\le \e_\eta^{-1},\quad \forall z\in \R^N,\ |z|>R_\eta. 
\eeq
Let $r>0$ be such that $|\alpha|_{N/p,\R^N\setminus B_{r_\eta}(0)}< {\eta}$ and $R_\eta>0$ such that 
$$
\supp \phi_{z,\e}\subseteq \R^N\setminus B_{r_\eta}(0)\qquad\forall \e\le\e_\eta,\ \forall z\in\R^N,\ |z|>R_\eta.
$$
Then \eqref{Sorrentino} follows by H\"older inequality.
\qed

\medskip

\medskip
  
\begin{lemma}
\label{L5.7}
Let  $a=a_\infty+\alpha$ as in \eqref{b}. If $a_\infty\not\equiv 0$ then 
\beq
\label{LL}
\cL(a_\infty):=\inf\{F_a(u)\ :\ u\in S_{p^*}(\R^N),\ \beta(u)=0,\ \gamma(u)\le\tfrac 12 \}>S.
\eeq
Moreover, for every fixed $\alpha\in L^{N/p}(\R^N)\setminus\{0\}$, $\alpha\ge 0$,
\beq
\label{sc}
\lim_{|a_\infty|_\infty\to 0}\cL(a_\infty)=S.
\eeq
\end{lemma}

\proof
Let $a_\infty>0$ and assume, by contradiction, that $\cL=S$.
Then, arguing as in the proof of Lemma \ref{L5.3}, we can find a sequence $(u_n)$ in $S_{p^*}(\R^N)$ such that
\begin{eqnarray}
\vspace{2mm}
&& \label{Sus}
\lim_{n\to\infty}F_a(u_n)= \lim_{n\to\infty}F_0(u_n) =S,\\
\vspace{2mm}
&&\label{Sus1}\beta(u_n)=0,\ \gamma(u_n)\le\tfrac 12,\  \forall n\in\N, \\
&&\label{Sus2}u_n(x)=U_{y_n,\e_n}(x)+\Phi_n(x),\quad \ \mbox{ with }\ \Phi_n\to 0\ \mbox{ in }\cD^{1,p}(\R^N),
\end{eqnarray}
for suitable $y_n\in\R^N$ and $\e_n>0$. 
We claim that, up to a subsequence,  
\beq
\label{Scho}
(a)\ \lim_{n\to\infty}\e_n=\bar\e>0,\qquad (b)\ \lim_{n\to\infty}y_n=0,
\eeq
 yielding a contradiction as in \eqref{Clapp}. 

 \smallskip
 
To prove \eqref{Scho} we can proceed exactly as in the proof of Lemma \ref{L5.3}, except where the case $\e_n\to\infty$ is ruled out.
Here, instead, we observe that $\e^N|u_n(\e_n(x-y_n))|^{p^*}\to U^{p^*}_{0,1}$ a.e. in $\R^N$ and 
$$
F_a(u_n)\ge S+a_\infty\int_{B_{\e_n}(y_n)}u_n^pdx.
 $$
Then, if $\e_n\to\infty$, by Fatou  lemma we have
$$
\liminf_{n\to\infty}F_a(u_n)\ge S +a_\infty\liminf_{n\to\infty}\e_n^{p}\int_{B_1 (0)}U^p_{0,1}dx=\infty,
$$
contrary to  \eqref{Sus}. 
So, the proof of \eqref{LL} is complete.

\smallskip

In order to prove \eqref{sc}, we fix $\eta>0$ and choose a radial function $\bar\phi\in S_{p^*}(\R^N)\cap\cC^\infty_0(\R^N)$ such that $F_0(\bar\phi)<S+\tfrac\eta4.$
Then, by Lemma \ref{lemma} and Proposition \ref{Pcet}, we can consider a $\bar\phi_{0,\e_\eta}$ such that
\beq
\label{tri}
\beta( \bar\phi_{0,\e_\eta})=0,\quad\gamma( \bar\phi_{0,\e_\eta})<\tfrac12,\qquad F_a( \bar\phi_{0,\e_\eta})<S+\frac\eta2+ |a_\infty|_\infty \int_{\R^N}  \bar\phi_{0,\e_\eta}^pdx.
\eeq
Since $\lim\limits_{|a_\infty|_\infty\to 0}|a_\infty|_\infty \int_{\R^N} \bar\phi_{0,\e_\eta}^pdx=0$,  from \eqref{tri} we infer \eqref{sc}.
\qed

\medskip

In the compactness range we are working on, the critical points we find are constant sign functions, at it is shown by the following
\begin{prop}
\label{Psegno}
Let $\Omega\subseteq\R^N$ and  $a$   satisfies \eqref{b}.
If $u\in S_{p^*}(\R^N)$ is a critical point of $F_a$ constrained on  $S_{p^*}(\R^N)$ and $u^+,u^-\not\equiv 0$ then $F_a(u) > 2^{p/N}S$.
\end{prop}

\proof
By Proposition \ref{NEM}
\beq
\label{eCP1}
F_a(u_{\pm}) > S |u_\pm|_{p^*}^p.
\eeq
Since $u$ is a constrained critical point, then 
\beq
\label{eCP2}
F_a'(u)[v]=\lambda_u \int_\Omega |u|^{p^*-2}u\, v\, dx, \qquad \forall v\in X,
\eeq
for a suitable Lagrange multiplier $\lambda_u$.
Testing \eqref{eCP2} with $u$ we see $\lambda_u=F_a(u)$, whereas by testing it with $u_\pm$ we get
\beq
\label{eCP3}
F_a(u_\pm)=F_a(u)|u_\pm|_{p^*}^{p^*}.
\eeq
From \eqref{eCP1} and \eqref{eCP3} we obtain that
$$
|u_\pm|_{p^*}^{p^*}> \left(\frac{S}{F_a(u)}\right)^{N/p}.
$$
Hence, our claim follows from
$$
F_a(u)=F_a(u_+)+F_a(u_-)=F_a(u)(|u_+|_{p^*}^{p^*}+|u_-|_{p^*}^{p^*})>2\, F_a(u)\left(\frac{S}{F_a(u)}\right)^{N/p}.
$$
\qed

  

\sezione{The case $\Omega=\R^N$: proof of the existence Theorem \ref{R^N}}\label{unboundedsection}
 
\noindent {\em Proof of Theorem \ref{R^N}}\qquad
{\bf{\em (i)}}\qquad Let $\bar R$ as in Proposition \ref{Pcet} and define 
$$
\cA(a_\infty):=\max\{F_a(\phi_{z,\e})\ :\ (z,\e)\in Q_{\bar R}\}.
$$
By H\"older inequality, and taking into account \eqref{BP},  we can find $\Lambda>0$ such that $\cA(0)<2^{p/N}S$ for every $\alpha\in L^{N/p}(\R^N)$ such that $|\alpha|_{N/p}<\Lambda$.
Moreover, by   Lemma \ref{lemma}{\em (i)}, \eqref{Pac} and the continuity of the map $\e\to\gamma(\phi_{0,\e})$ one can easily see that there exists $\mu\in (\bar R^{-1},\bar R)$ such that $\gamma(\phi_{0,\mu})=1/2$, besides $\beta(\phi_{0,\mu})=0$, hence
\beq
\label{We}
S<\cB(0,\alpha)\le\cA(0)<2^{p/N}S.
\eeq
 
 Taking into account Theorem \ref{comp} and \eqref{We}, we can proceed exactly as in the proof of Theorem 1.1. in \cite[p.24]{CM19ESAIM} and get a critical value $c_h$ for $F_a$ constrained on $S_{p^*}(\R^N)$, in the energy range $[\cB(0,\alpha_1), \cA(0,\alpha_1)]$. 
 The critical value $c_h$ provides a solution $u_h$, that is nonnegative by Proposition \ref{Psegno}.
 
 \bigskip

{\bf{\em (ii)}}\qquad By Lemma \ref{L5.7} and Proposition \ref{Pcet}, if $|a_\infty|_\infty>0$ is sufficiently small, then
\beq
S<\cL(a_\infty,)\le \cG(a_\infty)<\min\{\cB(a_\infty,\alpha_1),2^{p/N}S\}.
\eeq
Then, taking into account Corollary \ref{comp}, we can proceed as in the proof of Theorem 1.2. in \cite[p.27]{CM19ESAIM} and get a critical value $c_l$ for $F_a$ constrained on $S_{p^*}(\R^N)$, in the energy range $[\cL(a_\infty),\cG(a_\infty)]$, that provides  a solution $u_l$ that is non-negative by Proposition \ref{Psegno}.
 
 \bigskip

{\bf{\em (iii)}}\qquad  Now, let us consider the case $a_\infty>0$ and $\alpha$ as in point $(i)$. 

Since $\phi$ and $\bar R$ are fixed, clearly $\lim_{|a_\infty|_\infty\to 0} \cA(\infty)=\cA(0)$.
Moreover, $\cB(0,\alpha)\le \cB (a_\infty,\alpha)$ for every $a_\infty$ as in \eqref{b}.
Then, proceeding as in $(i)$ and $(ii)$,  we get
\beq
\label{beq}
S<\cL(a_\infty)\le\cG(a_\infty)<\cB(0,\alpha)\le \cB(a_\infty,\alpha)\le\cA(a_\infty)<2^{p/N}S.
\eeq
Then by Theorem \ref{comp} one can argue exactly as in the proof of Theorems 1.1 and 1.2 in \cite{CM19ESAIM},  showing the existence of two distinct critical levels $c_l\in (\cL(a_\infty),\cG(a_\infty))$ and $c_h\in (\cB(a_\infty,\alpha),\cA(a_\infty))$.
So we obtain two non-negative solutions by Proposition \ref{Psegno}.

\qed



\sezione{The case $\Omega$ a bounded domain: proof of the existence Theorem \ref{bounded}} \label{boundedsection}

In the present section we use the same notation as in the previous one, working with functions in $W^{1,p}_0(\Omega)$ instead of $W^{1,p}(\R^N)$.

\begin{lemma}
\label{LVit}
Let $a$ as in \eqref{b} or $a\equiv 0$,  and $\Omega\subset\R^N$  a bounded domain, then
\beq
\label{Vit}
\cL(a_\infty)>S.
\eeq
Moreover,
\beq
\label{Vit1}
\cL(a_\infty)\to S\quad\mbox{ as }a_\infty\to 0\ \mbox{ and }\ \rho(\Omega)\to\infty,
\eeq
see \eqref{rho}. 
\end{lemma}
\proof
Let us prove \eqref{Vit}.
Arguing by contradiction, as in the proof of Lemma \ref{L5.7}, we get a sequence $(u_n)$ in  $W^{1,p}_0(\Omega)$ such that \eqref{Sus}--\eqref{Sus2} hold.
Using the same notation, we can easily see that $\e_n\to \bar\e\in (0,\infty]$ cannot occur. 
In fact, if this were not the case we would have, up to a subsequence, 
$$
1\equiv\int_\Omega |u_n|^{p^*}dx=\int_\Omega |U_{y_n,\e_n}+\Phi_n|^{p^*}dx\to c\in [0,1),
$$
a contradiction.
Then the only possible case is $\e_n\to 0$, that can be ruled out arguing as in \eqref{new}, \eqref{new1}.
So, \eqref{Vit} is justified.

The proof of \eqref{Vit1} follows by the proof of \eqref{sc}, taking into account that $\supp\bar\phi_{0,\eta}\subset\Omega$, when $\rho(\Omega)\to\infty$ (see \eqref{tri}).
\qed
 
  \bigskip

\noindent {\em Proof of Theorem \ref{bounded}}\qquad
{\bf{\em (i)}}\qquad Since the test function $\phi$  introduced in \eqref{BP} has compact support, Lemma \ref{Pcet} holds whenever $\rho(\Omega)$ is large enough.
If $|a_\infty|_\infty\in[0,M_1)$ then, by using Lemma \ref{LVit}, we have the following configuration of levels
$$
S<\cL(a_\infty)\le\cG(a_\infty)<\min\{\cB(a_\infty,\alpha_1),2^{p/N}S\}.
$$
So, the arguments used to prove Theorem \ref{R^N}{\em (ii)} ensure the existence of the low energy solution $u_l$.

 \medskip

{\bf{\em (ii)}}\qquad 
Let $|\alpha|_{N/p}<L$ and $|a_\infty|\in[0,M_2)$, then by \eqref{beq} we see that for large $\rho(\Omega)$ the following estimates hold
\beq
\label{We1Bis}
S<\cG(a_\infty)<\cB(a_\infty,\alpha)\le\cA(a_\infty)<2^{p/N}S. 
\eeq
 Then we can conclude as in Theorem \ref{R^N}{\em (iii)}
\qed




\begin{thebibliography}{10}

\bibitem{AlMeMo20AMPA}
W.~Albalawi, C.~Mercuri, and V.~Moroz.
\newblock Groundstate asymptotics for a class of singularly perturbed
  {$p$}-{L}aplacian problems in {$\mathbb{R}^N$}.
\newblock {\em Ann. Mat. Pura Appl. (4)}, 199(1):23--63, 2020.

\bibitem{MR1926623}
C.~O. Alves.
\newblock Existence of positive solutions for a problem with lack of
  compactness involving the {$p$}-{L}aplacian.
\newblock {\em Nonlinear Anal.}, 51(7):1187--1206, 2002.

\bibitem{AFM21DCDS}
C.~O. Alves, G.~M. Figueiredo, and R.~Molle.
\newblock Multiple positive bound state solutions for a critical {C}hoquard
  equation.
\newblock {\em Discrete Contin. Dyn. Syst.}, 41(10):4887--4919, 2021.

\bibitem{MR2292344}
A.~Ambrosetti and A.~Malchiodi.
\newblock {\em Nonlinear analysis and semilinear elliptic problems}, volume 104
  of {\em Cambridge Studies in Advanced Mathematics}.
\newblock Cambridge University Press, Cambridge, 2007.

\bibitem{BeCe90JFA}
V.~Benci and G.~Cerami.
\newblock Existence of positive solutions of the equation {$-\Delta
  u+a(x)u=u^{(N+2)/(N-2)}$} in {${\bf R}^N$}.
\newblock {\em J. Funct. Anal.}, 88(1):90--117, 1990.

\bibitem{MR0784102}
H.~Brezis and J.-M. Coron.
\newblock Convergence of solutions of {$H$}-systems or how to blow bubbles.
\newblock {\em Arch. Rational Mech. Anal.}, 89(1):21--56, 1985.

\bibitem{MR0699419}
H.~Br\'{e}zis and E.~Lieb.
\newblock A relation between pointwise convergence of functions and convergence
  of functionals.
\newblock {\em Proc. Amer. Math. Soc.}, 88(3):486--490, 1983.

\bibitem{BrNi83CPAM}
H.~Br\'{e}zis and L.~Nirenberg.
\newblock Positive solutions of nonlinear elliptic equations involving critical
  {S}obolev exponents.
\newblock {\em Comm. Pure Appl. Math.}, 36(4):437--477, 1983.

\bibitem{BrWi08JFA}
H.~Br\'ezis and M.~Willem.
\newblock On some nonlinear equations with critical exponents.
\newblock {\em J. Funct. Anal.}, 255(9):2286--2298, 2008.

\bibitem{CeDeSo05CVPDE}
G.~Cerami, G.~Devillanova, and S.~Solimini.
\newblock Infinitely many bound states for some nonlinear scalar field
  equations.
\newblock {\em Calc. Var. Partial Differential Equations}, 23(2):139--168,
  2005.

\bibitem{CM19ESAIM}
G.~Cerami and R.~Molle.
\newblock Multiple positive bound states for critical
  {S}chr\"{o}dinger-{P}oisson systems.
\newblock {\em ESAIM Control Optim. Calc. Var.}, 25:Paper No. 73, 29, 2019.

\bibitem{CePa03CVPDE}
G.~Cerami and D.~Passaseo.
\newblock The effect of concentrating potentials in some singularly perturbed
  problems.
\newblock {\em Calc. Var. Partial Differential Equations}, 17(3):257--281,
  2003.

\bibitem{CeSoSt86JFA}
G.~Cerami, S.~Solimini, and M.~Struwe.
\newblock Some existence results for superlinear elliptic boundary value
  problems involving critical exponents.
\newblock {\em J. Funct. Anal.}, 69(3):289--306, 1986.

\bibitem{CiPi04ZAMP}
S.~Cingolani and A.~Pistoia.
\newblock Nonexistence of single blow-up solutions for a nonlinear
  {S}chr\"{o}dinger equation involving critical {S}obolev exponent.
\newblock {\em Z. Angew. Math. Phys.}, 55(2):201--215, 2004.

\bibitem{MR2560133}
S.~de~Valeriola and M.~Willem.
\newblock On some quasilinear critical problems.
\newblock {\em Adv. Nonlinear Stud.}, 9(4):825--836, 2009.

\bibitem{Db83NA}
E.~DiBenedetto.
\newblock {$C^{1+\alpha }$} local regularity of weak solutions of degenerate
  elliptic equations.
\newblock {\em Nonlinear Anal.}, 7(8):827--850, 1983.

\bibitem{MR0688279}
M.~J. Esteban and P.-L. Lions.
\newblock Existence and nonexistence results for semilinear elliptic problems
  in unbounded domains.
\newblock {\em Proc. Roy. Soc. Edinburgh Sect. A}, 93(1-2):1--14, 1982/83.

\bibitem{FaMeWi19CVPDE}
A.~Farina, C.~Mercuri, and M.~Willem.
\newblock A {L}iouville theorem for the {$p$}-{L}aplacian and related
  questions.
\newblock {\em Calc. Var. Partial Differential Equations}, 58(4):Paper No. 153,
  13, 2019.

\bibitem{GaSiYaZh20PRSE}
F.~Gao, E.~D. da~Silva, M.~Yang, and J.~Zhou.
\newblock Existence of solutions for critical {C}hoquard equations via the
  concentration-compactness method.
\newblock {\em Proc. Roy. Soc. Edinburgh Sect. A}, 150(2):921--954, 2020.

\bibitem{GuVe88JDE}
M.~Guedda and L.~V\'{e}ron.
\newblock Local and global properties of solutions of quasilinear elliptic
  equations.
\newblock {\em J. Differential Equations}, 76(1):159--189, 1988.

\bibitem{GuVe89NA}
M.~Guedda and L.~V\'{e}ron.
\newblock Quasilinear elliptic equations involving critical {S}obolev
  exponents.
\newblock {\em Nonlinear Anal.}, 13(8):879--902, 1989.

\bibitem{GLLM22arXiv}
L.~Guo, Q.~Li, X.~Luo, and R.~Molle.
\newblock Standing waves for two-component elliptic system with critical growth
  in $\mathbb{R}^{4}$: the attractive case.
\newblock {\em arXiv:2211.03425}, 2022.

\bibitem{MR0834360}
P.-L. Lions.
\newblock The concentration-compactness principle in the calculus of
  variations. {T}he limit case. {I}.
\newblock {\em Rev. Mat. Iberoamericana}, 1(1):145--201, 1985.

\bibitem{MR0850686}
P.-L. Lions.
\newblock The concentration-compactness principle in the calculus of
  variations. {T}he limit case. {II}.
\newblock {\em Rev. Mat. Iberoamericana}, 1(2):45--121, 1985.

\bibitem{MeWi10DCDS}
C.~Mercuri and M.~Willem.
\newblock A global compactness result for the {$p$}-{L}aplacian involving
  critical nonlinearities.
\newblock {\em Discrete Contin. Dyn. Syst.}, 28(2):469--493, 2010.

\bibitem{Pa96AIHP}
D.~Passaseo.
\newblock Some sufficient conditions for the existence of positive solutions to
  the equation {$-\Delta u+a(x)u=u^{2^*-1}$} in bounded domains.
\newblock {\em Ann. Inst. H. Poincar\'{e} C Anal. Non Lin\'{e}aire},
  13(2):185--227, 1996.

\bibitem{PeWaYa18JFA}
S.~Peng, C.~Wang, and S.~Yan.
\newblock Construction of solutions via local {P}ohozaev identities.
\newblock {\em J. Funct. Anal.}, 274(9):2606--2633, 2018.

\bibitem{Peral97Notes}
I.~Peral.
\newblock {\em Multiplicity of Solutions for the p-Laplacian}.
\newblock Lecture notes for the Second School of Nonlinear Functional Analysis and Applications
  to Differential Equations. International Centre of Theoretical Physics -
  Trieste (Italy), 1997.

\bibitem{MR0855181}
P.~Pucci and J.~Serrin.
\newblock A general variational identity.
\newblock {\em Indiana Univ. Math. J.}, 35(3):681--703, 1986.

\bibitem{MR2356201}
P.~Pucci and J.~Serrin.
\newblock {\em The maximum principle}, volume~73 of {\em Progress in Nonlinear
  Differential Equations and their Applications}.
\newblock Birkh\"{a}user Verlag, Basel, 2007.

\bibitem{Sc16AM}
B.~Sciunzi.
\newblock Classification of positive
  {$\mathcal{D}^{1,p}(\mathbb{R}^N)$}-solutions to the critical {$p$}-{L}aplace
  equation in {$\mathbb{R}^N$}.
\newblock {\em Adv. Math.}, 291:12--23, 2016.

\bibitem{St84MZ}
M.~Struwe.
\newblock A global compactness result for elliptic boundary value problems
  involving limiting nonlinearities.
\newblock {\em Math. Z.}, 187(4):511--517, 1984.

\bibitem{Ta76AMPA}
G.~Talenti.
\newblock Best constant in {S}obolev inequality.
\newblock {\em Ann. Mat. Pura Appl. (4)}, 110:353--372, 1976.

\bibitem{MR0727034}
P.~Tolksdorf.
\newblock Regularity for a more general class of quasilinear elliptic
  equations.
\newblock {\em J. Differential Equations}, 51(1):126--150, 1984.

\bibitem{MR0240748}
N.~S. Trudinger.
\newblock Remarks concerning the conformal deformation of {R}iemannian
  structures on compact manifolds.
\newblock {\em Ann. Scuola Norm. Sup. Pisa Cl. Sci. (3)}, 22:265--274, 1968.

\bibitem{Ve16JDE}
J.~V\'{e}tois.
\newblock A priori estimates and application to the symmetry of solutions for
  critical {$p$}-{L}aplace equations.
\newblock {\em J. Differential Equations}, 260(1):149--161, 2016.

\bibitem{MR1400007}
M.~Willem.
\newblock {\em Minimax theorems}, volume~24 of {\em Progress in Nonlinear
  Differential Equations and their Applications}.
\newblock Birkh\"{a}user Boston, Inc., Boston, MA, 1996.

\end{thebibliography}
\end{document}